\documentclass[11pt]{amsart}
\usepackage{amsfonts, latexsym, amsmath,amstext,amsthm,amssymb,amsxtra}
\usepackage{verbatim, enumerate}

\newtheorem{thm}{Theorem}

\newtheorem*{clm*}{Claim}

\newtheorem{lem}{Lemma}[section]
\newtheorem{rmk}{Remark}[section]
\newtheorem{prop}{Proposition}[section]

\newtheorem{obs}{Observation}

\newtheorem{cor}{Corollary}[section]
\newtheorem*{cor*}{Corollary}

\newcommand{\ds}{\displaystyle}

\newcommand{\R}{\mathbb R}
\newcommand{\C}{\mathbb C}
\newcommand{\E}{\mathbb E}
\newcommand{\N}{\mathbb N}
\newcommand{\Z}{\mathbb Z}

\newcommand{\Pro}{\mathbb P}
\newcommand{\lm}{\lambda}
\newcommand{\si}{\sigma}
\newcommand{\g}{\gamma}

\newcommand{\ep}{\varepsilon}
\newcommand{\M}{\mathcal{M}^{+}(T^{*})}
\newcommand{\calN}{\mathcal{N}}

\newcommand{\al}{\alpha}

\newcommand{\ind}{1{\hskip -2.5 pt}\hbox{I}}

\title[Gap probability for Gaussian Stationary Processes]
{Long gaps between sign-changes of Gaussian
Stationary Processes}

\author[N. D. Feldheim and O. N. Feldheim]{Naomi D. Feldheim$^{1,2}$ \and Ohad N. Feldheim$^{1,3}$}

\begin{document}

\maketitle
\begin{abstract}
We study the probability of a real-valued stationary process to be positive on a large interval $[0,N]$. We show that if in some neighborhood of the origin the spectral measure of the process has density which is bounded away from zero and infinity, then the decay of this probability is bounded between two exponential functions in $N$. This generalizes similar bounds obtained for particular cases, such as a recent result by Artezana, Buckley, Marzo, Olsen.
\end{abstract}

\footnotetext[1]{School of Mathematical Sciences, Tel Aviv University, Tel Aviv, Israel.}
\footnotetext[2]{E-mail: trinomi@gmail.com. Research supported by the Science Foundation of the Israel Academy of Sciences and Humanities, grant 166/11.}
\footnotetext[3]{E-mail: ohad\_f@netvision.net.il. Research supported by an ERC advanced grant.}

\section{Introduction}

\subsection{Definitions}
Let $T$ be either $\Z$ or $\R$, with the usual topology. A
\emph{Gaussian process} (GP) on $T$ is a random function $f:T\to \R$
whose finite marginals, that is $(f(t_1),\dots,f(t_n))$ for
      any $t_1,\dots,t_n\in T$, have multi-variate Gaussian distribution.
A GP on $\Z$ is called a \emph{Gaussian sequence}, while a GP on
$\R$ is called a \emph{Gaussian function}. In what follows, we
always assume continuity of Gaussian functions.

A GP on $T$ whose distribution is invariant with respect to shifts
by any element of $T$, is called \emph{stationary}. We abbreviate
GSP, GSS and GSF for Gaussian stationary processes, sequences and
functions respectfully.

For a GSP $f$ on $T$ define the \emph{covariance function} $r: \
T\to\R$ as
$$r(t)=\E(f(0) f(t)).$$
Observe that due to stationarity, for every $t,s\in T$ we have
$$\E \left[f(s) f(t)\right]=r(t-s). $$
It is not difficult to verify that $r(\cdot)$ is a positive-definite
continuous function (see Adler and Taylor ~\cite[Chapter
1]{AdlTay}). By Bochner's theorem, there is a finite non-negative
measure $\rho$ on $T^*$ such that
$$r(t) = \widehat {\rho}(t):=\int_{T^*} e^{-i\lm t}d\rho(\lm). $$
Here $T^*$ is the dual of $T$, i.e. $\Z^*\simeq [-\pi,\pi]$ and
$\R^*\simeq\R$. We use the notation $\M$ for the set of all finite
non-negative measures on $T^{*}$. The measure $\rho= \rho_f\in \M$
is called \emph{the spectral measure} of the process $f$. Notice
that $\rho$ must be symmetric, i.e., for any interval $I$:
$\rho(-I)=\rho(I)$. Any $\rho\in \M$ uniquely defines a GSP $f$.

Throughout the paper, we shall assume the following condition:
\begin{equation}\label{eq: cond moment}
  \exists \delta>0: \: \int_{T^*} |\lm|^\delta d\rho(\lm) <\infty.
\end{equation}
This condition is enough to ensure that the associated process $f$
will be continuous (see once again ~\cite[Chapter 1]{AdlTay}).
Notice that this holds trivially in case $T=\Z$.

\subsection{Results}
Let $f:T\to\R$ be a GSP. Define the "gap probability" of $f$ to be
$$H_f(N)=\Pro\left(\forall t\in[0,N)\cap T: \ f(t)>0\right), $$
where $N\in\R $ is a parameter. This describes the probability that
no sign-changes of $f$ occurred in a time interval of length $N$. We
study the asymptotics of this probability as $N\to \infty$. It makes
no essential difference to regard $N$ as an integer, and we usually
do so.

Our main results are the following. Let $f$ be a Gaussian stationary
process on $T=\Z$ or $T=\R$, with spectral measure
  $\rho\in \M$, satisfying ~\eqref{eq: cond moment}.

\begin{thm}[upper bound]\label{thm: upper bd gen}
Suppose that there exists $a>0$ and two positive numbers $M,m>0$
such that
$$ \text{for any interval } I\subset (-a,a), \:\: m|I|\le \rho(I)\le M|I|.$$
Then there exists $C=C(a,m,M)>0$ such that for all large enough $N$,
$$H_f(N)\le e^{-C N}.$$
\end{thm}

\begin{thm}[lower bound]\label{thm: lower bd gen}
Suppose that there exists $a>0$ and a number $m>0$ such that
$$ \text{for any interval } I\subset (-a,a), \:\: m|I|\le \rho(I).$$
Then there exists $c=c(a,m)>0$ such that for all large enough $N$,
$$H_f(N)\ge e^{-cN}.$$
\end{thm}

\begin{rmk}\label{rmk: up extend}
{\rm The condition in Theorem~\ref{thm: upper bd gen} may be
replaced by the following: There exist two intervals $J_1=(-a,a)$
and $J_2$, and two numbers $M,m>0$, such that
\begin{enumerate}[(i)]
  \item\label{item: up bd} for any interval $I\subset J_1$: $\rho(I)\le M|I|$, and
  \item\label{item: low bd} for any interval $I\subset J_2$: $m|I|\le \rho(I)$.
\end{enumerate}
The necessary changes in the proof are indicated in Section
\ref{sec: up-extend}. However, the authors believe condition
~\eqref{item: up bd} might be enough to ensure an upper exponential
bound on $H(N)$. }
\end{rmk}

\begin{rmk}\label{rmk: fast slow}
{\rm Examples for which $H(N)$ tends to zero slower than any
exponential in $N$ are known; Newell and Rosenblatt construct one in
~\cite{NR}.

Examples for which $H(N)$ tends to zero faster than any exponential
in $N$ are also known. A simple example was pointed out to us by M.
Krishnapur. Let $(Y_j)_{j\in\Z}$ be a GS with independent entries,
and define $X_j=Y_j-Y_{j-1}$ for all $j\in \Z$. Then $X$ is a GSS
with $H_X(N) = \frac 1{N!} \simeq e^{-C N\log N}$, for a suitable
constant $C>0$. Notice that the spectral measure has density
$2(1-\cos(\lm))$, $\lm \in [-\pi,\pi]$, which vanishes at $\lm=0$.}
\end{rmk}

\subsection{Overview}
The rest of the paper is organized as follows. Section~\ref{sec:
discuss} is devoted to discussion of the results. This includes an
historical background, and a simple yet useful observation that we
shall use (Observation~\ref{obs: rho1+rho2} below). The results are
then proved independently: Theorem~\ref{thm: upper bd gen} (an upper
exponential bound) is proved in Section~\ref{sec: up}, while
Theorem~\ref{thm: lower bd gen} (a lower exponential bound) is
proved in Section~\ref{sec: low}.

\subsection{Acknowledgements}
We thank Mikhail Sodin for introducing
us to the problem and for his advice throughout the research. We are
grateful to Ron Peled for a conversation which laid the foundations
to Theorem~\ref{thm: upper bd gen}. Discussions with Jeremiah
Buckley, Amir Dembo, Manjunath Krishnapur, Zakhar Kubluchko,
Jan-Fredrik Olsen and Ofer Zeitouni improved our understanding of
the problem, its applications and its relation to other works.

\section{Discussion}\label{sec: discuss}
\subsection{Background}
Gap probability, sometimes referred to by the name "persistence
probablity" or "hole probability", was studied extensively in the
1960's, by Slepian \cite{Slep}, Longuet-Higgins~\cite{LonHig},
Newell-Rosenblatt~\cite{NR} and others. In addition to proving some
bounds and inequlities (such as the well-known "Slepian
inequality"), they developed series expansions which approximate
this probability quite well for small intervals. In a few examples,
exact expressions for the gap probability were calculated (see
~\cite{Slep} and references therein).

In the last decade or two, physicists (such as Majumdar-Bray
~\cite{phys1} and Ehrhardt-Majumdar-Bray ~\cite{phys2}) proposed
some new methods of approximation, especially for the long-range
regime. Their predictions suggest that in many cases of interest the
gap probability $H(N)$ behaves asymptotically like $e^{-\theta N}$,
with some $\theta>0$. A rigorous derivation of such a result is
still lacking.

In case the covariance function $r(t)$ is non-negative, Dembo and
Mukherjee~\cite[Theorem 1.6]{DemMu} proved those predictions are
correct; namely, that the limit
$$\lim_{N\to\infty} \frac {-\log H_f(N)}{N}$$
exists (possibly infinite).
The case when $r(t)$ changes sign, as well as computation of the
limit, remain open. We note that the work last mentioned, along with
other works by physicists such as Schehr-Majumdar~\cite{ScMa}, draw
connections between gap probabilities of GSPs, those of
diffusion processes, and those of zeros of random polynomials.

In this work we are interested in the case where $r(t)$ changes sign.
A simple and interesting example is the cardinal sine covariance $r(t)=\frac{\sin (\pi t)}{t}$,
which corresponds to indicator spectral density $\ind_{[-\pi,\pi]}$. In
an elegant recent work, Antezana, Buckley, Marzo and Olsen ~\cite{Barc} give
exponential upper and lower bounds for $H_f(N)$ (see
Theorem~\ref{thm: barc} below). Our research may be viewed as a an
extension of their result to other \emph{stationary} Gaussian
processes. Recently Antezana, Marzo and Olsen were able to
generalize this same result in the direction of Gaussian analytic
functions over \emph{de-Branges spaces} ~\cite{AMO}.

Via private communication we learned of results by
Krishnapur-Maddaly regarding lower bounds for the gap
probability of a SGS. It seems that our conditions for a lower
exponential bound are currently stronger, but they have given very
mild conditions which ensure $H_f(N)\ge e^{-cN^2}$ (where $c>0$ is a
constant, and the inequality holds for large ehough $N$). Though the
results are similar in spirit, their methods seem to be very
different from ours.

Lastly we mention an analogous result for the planar Gaussian
analytic function
$$\sum_{n\in \Z}a_n \frac{z^n}{\sqrt{n!}}, \text{  where } a_n\sim \calN_\C (0,1) \text{ are i.i.d.}$$
Bounds concerning hole probabilities for this model were obtained by
Sodin and Tsirelson~\cite{ST3}, and later refined by
Nishry~\cite{Alon}. They showed that the probability of having no
zeroes in a ball of radius $R$ in the plane is asymptotically
$e^{-(e^2/4+o(1))R^4}$, as $R\to\infty$. For discussion of such
results and comparison to other point processes in the plane, see
\cite[Chapter 7]{GAF book}.

\subsection{A Key Observation}\label{sec: obs}
We include here the basic observation which will be used to prove
both Theorems ~\ref{thm: upper bd gen} and ~\ref{thm: lower bd gen}.
We use the symbol $\oplus$ to indicate the sum of two independent
processes or random variables.

\begin{obs}\label{obs: rho1+rho2}
Let $f$ be a GSP on $T$ with spectral measure $\rho\in \M$, and
Suppose $\rho=\rho_1+\rho_2$, where $\rho_1, \rho_2 \in \M$. Then
the following equality holds in distribution:
$$f\overset{d}{=}f_1 \oplus f_2, $$
where $f_j$ is a GSP with spectral measure $\rho_j$ ($j=1,2$), and
$f_1$ is independent (as a process) from $f_2$.
\end{obs}

\begin{proof}
We calculate the covariance function of $f_1\oplus f_2$ using the
independence of the processes:
\begin{align*}
\E \Big[\left(f_1(0)+f_2(0)\right)\left(f_1(t)+f_2(t)\right) \Big]
&= \E f_1(0)f_1(t) + \E f_2(0)f_2(t) \\
&=\widehat{\rho_1}(t)+\widehat{\rho_2}(t) = \widehat{\rho}(t).
\end{align*}
This covariance function is equal to that of $f$. As all processes
are Gaussian, the observation follows.
\end{proof}

\section{Upper bound: proof of Theorem~\ref{thm: upper bd gen}}\label{sec: up}
This section is devoted to the proof of Theorem~\ref{thm: upper bd
gen}.

Let $f$ be a GSF or GSS with spectral measure $\rho$, obeying the
conditions of Theorem~\ref{thm: upper bd gen}. Let $k\in\N$ be such
that $\frac \pi k \le a$, and denote $J:=[-\pi / k, \pi /k]\subset
[-a,a]$. We decompose the spectral measure as follows:
$$d\rho(\lm)=m\ind_{J}(\lm)d\lm + d\mu(\lm),$$
where $\mu\in \M$ is non-negative and there
exists $M'>0$ such that
\begin{equation}\label{eq: cond bound on mu}
\text{for any interval } I \subset (-a,a): \: \mu(I)\le M'|I|.
\end{equation}

By Observation~\ref{obs: rho1+rho2}, we may represent
$$f\overset{d}{=}S\oplus g$$
where $S$ and $g$ are independent processes, with spectral measures
$m\ind_{J}(\lm)$ and $\mu$ respectively.

Next, we observe that sampling $S$ in a certain lattice results in
independent random variables:

\begin{obs}[indicator spectrum]\label{obs: ind spec}
The GSP $(S(t))_{t\in T}$ having spectral density
\emph{$m\ind_{[-\pi / k, \pi /k]}$} has the property that
$(S(jk))_{j\in\Z}$ are i.i.d. Gaussian random variables.
\end{obs}

\begin{proof}
By taking the Fourier transform of the given measure, the covariance
function of $S$ is
$$\E\left[ S(s) S(t) \right]= \frac {\sin(\pi (t-s)/k)}{(t-s)}. $$
Thus $S(jk)$ and $S(mk)$ are uncorrelated for any $j,m\in\Z$, $j\neq
m$; as these are Gaussian random variables - independence follows.
\end{proof}

In order to apply Observation~\ref{obs: ind spec}, we look at a
certain translated lattice $\{jk+l: j\in \Z\}$ on which $S$ is
indeed independent. The translation (which we call "split") of the
sampled lattice will depend on $g$.

More precisely, fix a number $q>0$ (say, $q=1$), and define an event
$E$ depending only on the process $(g(t))_{t\in T}$ in the following
way:
\begin{align*}
\ds
E&=\left\{\frac 1 N \sum_{t=1}^N g(t) < q \right\}, \text{ if $g$ is GSS}\\
E&=\left\{\frac 1 N \int_{0}^N g(t) dt < q \right\}, \text{ if $g$ is GSF}
\end{align*}
Using the law of total probability we have:
\begin{align*}
  \ds \Pro &\left( f(t)=S(t)+g(t) >0, \ 0< t\le N\right) \notag \\
  &\le   \Pro\left(S(t) +g(t)>0, \ 0< t\le N \: \Big|\: E \right) +
  \Pro\left(E^c \right).
\end{align*}
It is enough to show that there exist $C_1,C_2>0$ such that for
large enough $N$,
\begin{enumerate}[(I)]
\item~\label{item: given E} $\Pro\left(S(t) +g(t)>0, \ 0< t\le N \: \Big|\: E \right)
\le e^{-C_1 N}$, and
\item~\label{item: E^c} $ \Pro\left(E^c \right)\le e^{-C_2 N}$.
\end{enumerate}
We proceed the proof for the function-case, noting the sequence-case
follows similar lines and is generally easier.

We begin by showing ~\eqref{item: given E}.
It is enough to show
that there is $C_1>0$ such that for any large enough $N$ and any
fixed $g\in E$,
\begin{equation*}
\Pro\left(S(t) +g(t)>0, \ 0< t\le N \right) \le e^{-C_1 N}.
\end{equation*}
Indeed, this would imply (using the independence of $g$ and $S$):
\begin{align*}
&\Pro\left(S(t) +g(t)>0, \ 0< t\le N |\  E \right)\\
&=\E\left(  \Pro\left(S(t) +g(t)>0, \ 0< t\le N \right) \ \big| \ E  \right) \le e^{-C_1 N},
\end{align*}
as required.

To that end, we use a
property which holds when the event $E$ occurs, stated below.
\begin{obs}\label{obs: split}
Let $g$ be a continuous function such that
$\frac 1 N \int_0^N g(t) dt <q$,
and assume $N\in \N$ is divisible by $k$,  then there exists a number $l\in[0,k)$ such that
$$\frac k N \sum_{j=0}^{N/k-1} g(jk+l)<q. $$
\end{obs}

\begin{proof}
  Else, for every $l\in[0,k)$ the reverse inequality holds.
  Integrating it over $l\in[0,k]$ yields a contradiction.
\end{proof}

Now, fix a
function $g\in E$.
We can find a special split $l_g$ whose existence
is guaranteed by Observation~\ref{obs: split}. Therefore:
\begin{align*}
&\Pro\left(S(t) +g(t)>0, \ 0< t\le N \right) \\
&\le\Pro\left(S(jk+l_g)+g(jk+l_g)>0, \ j=0,1,\dots, N/k -1\right),
\end{align*}
where $(S(jk+l))_{j\in\Z}$ are i.i.d Gaussians (whose variance is
independent of $l_g$), and $\frac k N
\sum_{j=0}^{N/k-1} g(jk+l_g)<q$.
The following inequality will give the desired
bound.

\begin{prop}\label{prop: Xj+aj>0}
  Let $X_1,\dots,X_N$ be i.i.d real centered Gaussian random variables, and let $q\in \R$.
  There is a constant $C_q>0$ such that
  for any numbers $b_1,\dots,b_N\in \R$ which obey $\frac 1 N\sum_{j=1}^N b_j<q$, the following holds:
  \[
  \Pro\left(X_j+b_j>0,\: 1\le j\le N\right)\le e^{-C_q N}.
  \]
\end{prop}

\begin{proof}
Without loss of generality assume $\text{var}(X_1)=1$. Denote by
$\Phi(b) =\Pro(X_1<b)$ the cumulative distribution function of
$X_1$. By symmetry, $\Phi(b)=\Pro(X_1>-b)$. Using the "i.i.d"
property of the variables $\{X_j\}_{j=1}^N$ we have:
\[
p=\Pro\left(X_j+b_j>0,\: 1\le j\le N\right)=\prod_{j=1}^N
\Pro\left(X_j>-b_j\right)=\prod_{j=1}^N \Phi(b_j).
\]
Taking logarithm and using the concavity and monotonicity of
$x\mapsto \log \Phi(x)$, we get:
\[
\log p = \sum_{j=1}^N \log \Phi(b_j)\le N \cdot\log
\Phi\left(\frac{\sum_j^N b_j}{N}\right)< N\cdot \log \Phi(q),
\]
and so $C_q=-\log \Phi(q)>0$ is the desired constant.
\end{proof}

In order to prove ~\eqref{item: E^c}, we shall use the following:
\begin{prop}\label{prop: II}
$\frac 1 N \int_0^N g(t) dt\sim \calN_\R(0,\si_N^2)$, where
$\si_N^2\le \frac {C_0}{N}$ for all $N\in\N$ and some constant
$C_0>0$.
\end{prop}
\begin{proof}
The normality of the given integral follows from general arguments
of convergence of Gaussian random variables.
We focus on the bound on its variance. Recall that $\mu$ denoted the
spectral measure of $g$. We calculate the variance:
\begin{align*}
  \si_N^2&=\frac 1 {N^2} \E\left(\int_0^N g(t) dt\right)^2
  = \frac 1 {N^2}\iint_{[0,N]^2} \E (g(t) g(s)) dt\ ds\\
  &=\frac 1 {N^2} \int_0^N \int_0^N \widehat{\mu}(t-s) dt\ ds
  =\frac 1 {N} \int_{|t|<N}\left(1-\frac {|t|}{N}\right) \widehat{\mu}(t)dt.
\end{align*}
The change in order of integration and expectancy in the first
equality is easily justified by use of Fubini's theorem.

The inverse Fourier transform of $(1-\frac {|t|}{N})\ind_{[-N,N]}(t)$ is given by 
$$K_N(\lm)=N\left(\frac {\sin(N\lm/2)}{N\lm/2}\right)^2 \le \min\left(N,\frac {\pi^2}{N\lm^2}\right)$$
Using first Plancherel's identity, and then condition~\eqref{eq: cond bound
on mu} on the boundness of $\mu$, we get:
\begin{align*}
\si_N^2 &= \frac 1 N \int_\R K_N(\lm) d\mu(\lm) \\
&\le \int_{|\lm|< \frac \pi N} d\mu(\lm) + \frac {\pi^2}{N^2}\left(\int_{\frac{\pi}{N}\le |\lm|<a} +\int_{|\lm|\ge a}\right)\frac 1 {\lm^2}d\mu(\lm)\\
& \le M'\cdot \frac {2\pi}{N} + \frac {\pi^2}{N^2}\left( M' \int_{\frac{\pi}{N}\le |\lm|<a} \frac {d\lm} {\lm^2} + \frac 1 {a^2} \mu(\left\{|\lm|>a\right\}) \right) \\
& \le \frac {C_0} N,
\end{align*}
where $C_0$ is a constant (depending on $\mu$).
\end{proof}

At last, we prove ~\eqref{item: E^c}. Denote by $\g$ a standard
Gaussian random variable (i.e., distributed $\calN(0,1)$). Using the
Proposition~\ref{prop: II} together with the well-known inequality
\begin{equation*}
\forall y>0: \:\:\Pro\left(\g >y\right) <\frac 1 {\sqrt{2\pi}y}
e^{-y^2/2},
\end{equation*}
we get:
\begin{align*}
  \Pro(E^c)=\Pro\left( \frac 1 N \int_{0}^N g(t) \ge q \right)& =
  \Pro(\si_N \cdot \g \ge q)=\Pro\left(\g \ge \frac q {\si_N}\right) \\
  &\le \frac 1 {\sqrt{2\pi} }\cdot \frac{\si_N}{q} e^{-\frac 1 2 \cdot
  \frac{q^2}{\si_N^2}}\\
  & \le \frac 1 q \sqrt{\frac { C_0}{2\pi N}}  e^{-\frac{q ^2}{2 C_0} N} \le e^{-C_2 N},
\end{align*}
for a suitable choice of $C_2>0$ (depending only on $q$ and $\mu$).
Theorem~\ref{thm: upper bd gen} is proved.
\qed

\subsection{Extension: Proof of Remark~\ref{rmk: up extend}}\label{sec: up-extend}
Remark~\ref{rmk: up extend} states a somewhat more general condition
under which the conclusion of Theorem~\ref{thm: upper bd gen} is
true. The proof is only a slight modification of the one presented.
First, choose $l,k\in \N$ so that
$$J:=\left[\frac {(2l-1)\pi}{k},\frac {2l\pi}{k}\right] \subset J_2\cup(-J_2) .$$
Now decompose the measure as follows:
$$d\rho(\lm)=m\ind_{J\cup -J}(\lm)d\lm + d\mu(\lm).$$
By the premise, $\mu\in \M$ obeys the boundedness condition~\eqref{eq:
cond bound on mu} (just as before). Applying Observation ~\ref{obs:
rho1+rho2} we get
$$f\overset{d}{=} S \oplus g,$$
where $S$ has spectral measure $m\ind_{J\cup -J}(\lm)d\lm$ and $g$
has spectral measure $\mu$. We define $E$ and strive to prove items
~\eqref{item: given E} and ~\eqref{item: E^c}. Item~\eqref{item:
E^c} follows from Proposition~\ref{prop: II} and the calculation
following it with no change. The only property used in order to
prove item~\eqref{item: given E} is the independence of
$(S(jk))_{j\in\Z}$ (i.e., Observation~\ref{obs: ind spec}). Let us
show this still holds.

One way to end the argument is by calculation of the Fourier
transform of $\ind_{J\cup -J}(\lm)d\lm$ and observing it vanishes at
$kj$, $j\in\Z$ (just as in the proof of Observation~\ref{obs: ind
spec}). We give here a more general argument, relying on two
observations:
\begin{obs}\label{obs: scale}
  Let $(f(t))_{t\in \R}$ be a GSF with spectral measure $\rho$,
  and $\al>0$. Then the GSF $x\mapsto f(\al x)$  has spectral measure $\rho_\al$, defined by
  $$\forall I\subset \R:\ \: \rho_\al(I)=\rho(\{x\in \R:\ \al x \in I\}) $$
\end{obs}

\begin{proof}
  $\E \left[f(\al t)f(\al s)\right] = \widehat{\rho}(\al(t-s)) = \widehat{\rho_\al}(t-s)$.
\end{proof}

\begin{obs}\label{obs: sample}
  If $(f(t))_{t\in \R}$ is a GSF with spectral measure $\rho$,
  then sampling the lattice $(f(j))_{j\in\Z}$
  has the folded spectral measure $\rho^*\in\mathcal{M^{+}}([-\pi,\pi])$ obtained by:
  $\rho^*(I) = \sum_{m\in \Z }\rho(I + 2\pi m)$.
\end{obs}

\begin{proof}
  $\rho^{*}$ is the unique measure in $\mathcal M^{+}([-\pi,\pi])$ such that $\widehat{\rho^{*}}(j)=\widehat{\rho}(j)$ for any $j\in\Z$.
\end{proof}

Combining the last two observations, we get that if $(S(t))_{t\in
T}$ has spectral density $m\ind_{J\cup -J}(\cdot)$, then the
spectral density of $(S(kj))_{j\in\Z}$ is
$m\ind_{[-\pi,\pi]}(\cdot)$. Now Observation~\ref{obs: ind spec}
leads to the desired conclusion.

\section{Lower bound: Proof of Theorem~\ref{thm: lower bd gen}}\label{sec: low}
\subsection{Reducing GSS to GSF}
Theorem~\ref{thm: lower bd gen} is easily reduced to the case of
functions, by noticing the following:
\begin{obs}\label{obs: Hf<Hx}
  Any finite measure $\rho\in \mathcal{M}^{+}([-\pi,\pi])$ generates a GSF
  $f$ and a GSS $X$. The distribution of $(X(j))_{j\in\Z}$ is the
  same as that of $(f(j))_{j\in\Z}$ (since their covariance functions coincide).
  Moreover, for any number $N$:
\begin{align*}
H_f(N)&=\Pro(f(x)>0, \: x\in [0,N)\cap \R) \\
&\le \Pro(f(j)>0, \: j\in
  [0,N)\cap \N)=H_X(N).
  \end{align*}
\end{obs}

Therefore, in order to bound $H_X(N)$ from below where $X$ is a GSS,
it is enough to bound $H_f(N)$ from below where $f$ is the GSF with
the same spectral measure as $X$.

\subsection{Proof for GSF}
Let $(f(t))_{t\in\R}$ be a GSF with spectral measure $\rho$, obeying
the condition of Theorem~\ref{thm: lower bd gen}. By scaling $f$
(and therefore scaling its spectral measure according to
Observation~\ref{obs: scale}, we may assume the condition is
satisfied with $a=\pi$.

Just as in the proof of Theorem~\ref{thm: upper bd gen}, we
decompose the spectral measure in the following manner:
$$d\rho=m\ind_{[-\pi,\pi]}(\lm)d\lm + d\mu.$$
Applying Observation~\ref{obs: rho1+rho2} we have
$$f\overset{d}{=}S\oplus g$$
where $S$ and $g$ are independent processes, and the spectral
measure of $S$ has density $m\ind_{[-\pi,\pi]}(\lm)$.

We have:
\begin{align}\label{eq: main low}
  H_f(N)&=\Pro\left(S(x)+g(x)>0,\:\: 0\le x< N\right) \notag \\
  & \ge \Pro\left(S(x)>d,\; 0\le x< N\right)\Pro\left(|g(x)|\le \frac d 2, \;0\le x< N\right),
\end{align}
where $d>0$ is a parameter of our choice. The first probability is
bounded from below by the following theorem:

\begin{thm} [Antezana, Buckley, Marzo, Olsen ~\cite{Barc}] \label{thm: barc}
Let $S(x)$ be the GSF with spectral measure {\rm
$d\rho(\lm)={\ind}_{[-\pi,\pi]}(\lm)d\lm$}. Then for any $d>0$ there
exists a constant $c_d>0$, such that for all $N\in\N$,
\[
\Pro\left(S(x)>d, \:\:0\le x< N\right) \ge e^{-c_d N}.
\]
\end{thm}

We turn to bound the second probability in ~\eqref{eq: main low},
i.e., the probability of the event $\{|g(x)|\le \ep, \:\: 0\le x< N
\}$. This is known in literature as a "small ball probability", and
is bounded from below by the following result:

\begin{lem}[Talagrand~\cite{Tal}, Shao and Wang ~\cite{ShaoWang}]\label{lem: ShW}
Let $(f(t))_{t\in I}$ be a centered Gaussian process on a finite
interval $I$. Suppose that for some $c>0$ and $0<\delta\le 2$,
\begin{equation*}
d_f(s,t)^2:= \E |f(s)-f(t)|^2 \le c|t-s|^{\delta}, \:\:s,t\in I.
\end{equation*}
Then, for some $K>0$ and every $\ep>0$,
$$\Pro\left(\sup_{t\in I}{|f(t)|\le \ep} \right)
\ge \exp\left(-\frac{K |I|}{\ep^{2/\delta}}\right).$$
\end{lem}

The proof of Lemma~\ref{lem: ShW}, apart from being deduced from a
much more general result in Talagrand's paper, may be found in notes
by Ledoux~\cite[Ch. 7]{Led} (but in a slightly different version).
Shao and Wang decided to omit a proof from their paper as they
learned that Talagrand's result generalizes theirs; but they do
include the most close formulation to the one above.

We draw the following corollary:
\begin{cor}\label{cor: moment small ball}
Let $f$ be a Gaussian stationary function on $\R$ with spectral
measure $\rho$, obeying the moment condition~\eqref{eq: cond
moment}. Then for all $\ep>0$ there exists $C, K>0$ such that for
any interval $I$ and any $N\in\N$:
$$\Pro\left(\sup_{I} |f| <\ep \right)\ge C e^{-K|I|}.$$
\end{cor}

Applying the corollary to $f=g$, $I=[0,N)$ and $\ep=\frac d 2>0$,
will give the desired bound on the second factor in ~\eqref{eq: main
low}, thus ending the proof of Theorem~\ref{thm: lower bd gen}.

\begin{proof}[Proof of Corollary ~\ref{cor: moment small ball}]
First we notice that if the moment condition~\eqref{eq: cond moment}
is satisfied with a certain exponent ~$\delta>0$, then it is also
satisfied by any smaller positive exponent. Therefore we may assume
$0<\delta<2$.

We shall check that $f$ obeys the condition of Lemma ~\ref{lem: ShW}
with this same ~$\delta$, i.e. that there exists a constant $c>0$
such that
\begin{equation*}
d_{f}(s,t)^2\le c|t-s|^{\delta}, \:\: s,t\in I.
\end{equation*}
Indeed:
\begin{align*}
  d_f(s,t)^2&= \E (f(s)-f(t))^2
  =2(r(0)-r(s-t))\\
  &=2\int_\R\Big(1-\cos(\lm( s-t)) \Big) d\rho(\lm)\le 2L|t-s|^\delta \int_\R |\lm|^\delta d\rho(\lm),
\end{align*}
where $L = \sup_{x\in\R} \frac {1-\cos(x)}{|x|^\delta}<\infty$. The
Corollary follows.
\end{proof}

\end{document}